\documentclass[12pt]{article}
\usepackage{amsfonts,amsthm}
\usepackage{amssymb,amsmath}
\usepackage{enumerate}
\usepackage{amsmath}

\usepackage[dvipsnames]{xcolor}
\usepackage{hyperref}
\hypersetup{
    colorlinks=true,
    citecolor=blue,
    linkcolor=blue,
    filecolor=magenta,
    urlcolor=blue,
    pdftitle={Overleaf Example},
    pdfpagemode=FullScreen,
    }

\input{amssym.def}

\newtheorem{Definition}{Definition}[section]
\newtheorem{Theorem}[Definition]{Theorem}
\newtheorem{Lemma}[Definition]{Lemma}
\newtheorem{Proposition}[Definition]{Proposition}
\newtheorem{Corollary}[Definition]{Corollary}
\newtheorem{Example}[Definition]{Example}

\newcommand{\be}{\begin{equation}}
\newcommand{\ee}{\end{equation}}

\begin{document}

\title{\bf On Quasi-Nil Clean Rings}
\author{\bf Saikat Das \footnote {e-mail : saikatofficial607@gmail.com} \ and ~\bf Sukhendu Kar  \footnote {e-mail : karsukhendu@yahoo.co.in} \\
{\small Department of Mathematics, Jadavpur University}\\
{\small 188, Raja S. C. Mallick Road, Kolkata - 700032, India.}}
\date{}
\maketitle

\begin{abstract}
In this paper, we introduce a new type of ring, called quasi-nil clean ring, where each element of the ring is the sum of a quasi-idempotent and a nilpotent element. We also investigate a particular class of ring, called strongly quasi-nil clean ring whose each element is the sum of a quasi-idempotent and a nilpotent element, where they commute. Our primary objective is to study the structural properties of these new class of rings, explore their relationships with existing classes of rings and establish some key characterizations. In the commutative setting, we provide a complete characterization in terms of quasi-Boolean quotients. Moreover we discuss about quasi-cleanness of amalgamated algebra and group ring.  \end{abstract}



\section{Introduction} \label{intro}

Ring theory plays a central role in algebra, with numerous generalizations of classical ring structures leading to significant theoretical developments. One of the key concepts in this area is the notion of clean rings, introduced by Nicholson in \cite{n1}. A ring is said to be clean if each of its elements can be expressed as the sum of a unit and an idempotent. Nicholson later introduced the concept of strongly clean ring in \cite{n2}. It is demonstrated in \cite{n2} that a strongly clean element of a ring satisfies a generalized form of Fitting's Lemma, leading to the conclusion that every strongly $\pi$-regular element of a ring is strongly clean.  The study of clean and strongly clean rings has been extensively studied in \cite{a1, c, d1, h}, inspiring various generalizations such as nil-clean rings and strongly clean rings, where the unit is replaced by a nilpotent element, introduced by Disel in \cite{d}. These variants have attracted considerable attention, especially due to their connections with other classes of rings, such as exchange rings, Boolean rings and strongly $\pi$-regular rings.

A recent development in this field is the introduction of quasi-clean rings by Tang et al \cite{qc} in $2023$. This class of rings extends the notion of clean rings by permitting each element to be expressed as the sum of a quasi-idempotent and a unit. Here, a quasi-idempotent is an element $a$ that satisfies $a^2 = ka$ for some central unit $k$ of the ring. This notion has been shown to extend many fundamental properties of clean rings while incorporating a wider range of algebraic structures. Motivated by these advancements, we introduce a new class of rings, called quasi-nil clean rings, which extend the notion of nil-clean rings. In a quasi-nil clean ring, every element is the sum of a quasi-idempotent and a nilpotent element. If, in addition, these two components commute, then the element is called strongly quasi-nil clean and if every element of a ring is strongly quasi-nil clean then the ring is called strongly quasi-nil clean.

The primary objective of this paper is to study the fundamental properties of quasi-nil clean rings and to explore their structural characteristics. The paper explores foundational properties of these rings, proving that every quasi-nil clean ring is necessarily quasi-clean and that its Jacobson radical is nil. It is shown that every finite direct product of quasi-nil clean rings is quasi-nil clean, while this does not necessarily hold for infinite products. The section also explores that quasi-nil clean rings are semipotent and we show that if an ideal $I$ of a strongly quasi-nil clean ring $R$ is a two-sided ideal, then any element of $I$ that is strongly quasi-nil clean in $R$ remains strongly quasi-nil clean in $I$. A particularly notable result is that for any idempotent $e^2=e$, the corner ring $eRe$ is strongly quasi-nil clean whenever $R$ itself is strongly quasi-nil clean and in the abelian case, all nilpotent elements lie in the Jacobson radical. Following this, the paper provides several characterizations of strongly quasi-nil clean rings.

The next section focuses on commutative quasi-nil clean rings, offering a detailed characterization in terms of their Jacobson radical and quotient structure. It is shown that a commutative ring $R$ is quasi-nil clean if and only if $J(R)$ is nil and the quotient ring $R/J(R)$ is quasi-Boolean (i.e., every element of the ring is quasi-idempotent). The paper further demonstrates that localization preserves quasi-nil cleanness in the commutative setting. Moreover in commutative UU-rings, the notions of quasi-clean, clean, nil-clean and quasi-nil clean are shown to be equivalent.

Then the investigation deals with amalgamated algebras along ideals particularly those of the form $A \bowtie^f B$, constructed via a ring homomorphism $f: A \to B$ and an ideal $K \subseteq B$. It is showed that if $A \bowtie^f B$ is quasi-nil clean, then so are both $A$ and $f(A) + K$. Furthermore, if the ideal $K$ is nil, then $A \bowtie^f B$ is quasi-nil clean if and only if $A$ is so. Lastly, this paper addresses group rings and their relation to quasi-nil clean structures. It provides sufficient conditions under which a group ring $RG$ inherits the quasi-nil cleanness. For instance, if $R$ is a quasi-nil clean ring with $2 \in U(R)$ and $G$ is a finite abelian elementary $2$-group, then the group ring $RG$ is quasi-nil clean. The analysis extends to commutative rings of characteristic $2^k$ and torsion $2$-groups, as well as cases where $R$ is quasi-nil clean and $p \in J(R)$, with $G$ an abelian locally finite $p$-group. In each case, the paper rigorously demonstrates how structural aspects of both the ring and group interact to preserve the quasi-nil cleanness.

Throughout this paper, we assume that a ring $R$ is associative with identity $1$. The symbols $J(R)$, $U(R)$, $Id(R)$, $N(R)$, $C(R)$ and $UC(R)$ denotes Jacobson radical of $R$, set of units of $R$, set of idempotents of $R$, set of nilpotents of $R$, center of $R$ and set of central units of $R$ respectively.

\section{Quasi-nil Clean Rings}
We say $a \in R$ is strongly regular when $a \in a^2R \cap Ra^2$. Equivalently, $a \in R$ is strongly regular if and only if $a=ue=eu$, where $e^2=e$ and $u$ is a unit in $R$ if and only if $a^2=ua=au$ for some unit $u$ in $R$. A quasi-idempotent is a special type of strongly regular element. An element $a\in R$ is said to quasi-idempotent if $q^2=kq$ for some central unit $k$ in $R$. It is observed that $q \in R$ is quasi-idempotent if and only if $q=ke$, where $k$ is a central unit and $e$ is an idempotent in $R$.

\begin{Definition}
     An element $a \in R$ is called quasi-nil clean if there is a quasi-idempotent $q$ and a nilpotent $n$ in $R$ such that $a=q+n$. Furthermore the quasi-nil clean element $a\in R$ is called strongly quasi-nil clean if they commute i.e. $qn=nq$. The ring $R$ is called (strongly) quasi-nil clean if each of its elements are  (resp. strongly) quasi-nil clean. \end{Definition}

\begin{Proposition}  \label{j}
    If $R$ is a quasi-nil clean ring, then $J(R)$ is nil.
\end{Proposition}
\begin{proof} Let $x \in J(R)$. Since $R$ is quasi-nil clean, we have $x=q+n$, where $q$ is a quasi-idempotent with $q^2=kq$ for some central unit $k$ and $n$ is a nilpotent. Now $k-q=k-(x-n)=k(1+k^{-1}n)-x=v-x$, where $v=k(1+k^{-1}n)$ is a unit in $R$. Therefore, $1-k^{-1}q=k^{-1}v(1-v^{-1}x)$ is a unit in $R$ and  $1-k^{-1}q$ is an idempotent in $R$. Therefore, $1-k^{-1}q=1$ which implies that $q=0$. This shows that $x$ is a nilpotent element in $R$. Hence $J(R)$ is nil. \end{proof}

\begin{Example}
$(i)$. Every nil-clean ring is necessarily a quasi-nil clean ring. However, the converse does not hold in general. For instance, the ring $\mathbb Z_6$  is quasi-nil clean but not nil-clean.

$(ii)$. Every homomorphic image of any (strongly) quasi-nil clean ring is (strongly) quasi-nil clean ring.

$(iii)$. Any finite direct product of quasi-nil clean rings is again quasi-nil clean.

$(iv)$. An infinite direct product of quasi-nil clean rings is not necessarily quasi-nil clean.  For example, consider the ring $R = \displaystyle \prod_{i=1}^\infty \mathbb{Z}_{2^i}$. Since $J(\mathbb{Z}_{2^i}) = 2\mathbb{Z}_{2^i}$ and $\mathbb{Z}_{2^i}/J(\mathbb{Z}_{2^i}) \cong \mathbb{Z}_2$, it follows by \cite[Theorem 2.7]{nc} that $\mathbb{Z}_{2^i}$ is strongly nil clean and consequently it is (strongly) quasi-nil clean. But
\[
\prod_{i=1}^\infty N(\mathbb{Z}_{2^i}) \subseteq \prod_{i=1}^\infty J(\mathbb{Z}_{2^i}) = J(R).
\]
As we have $2 \in N(\mathbb{Z}_{2^i})$ for each $i$, we have $(0,2,2,2,\dots) \in J(R)$. But the element $(0,2,2,\dots)$ is not nilpotent. This implies that $J(R)$ is not nil, which contradicts Proposition $\ref{j}$ i.e., $R$ is not a quasi-nil clean ring. \end{Example}

\begin{Proposition} \label{1.4}
 Every quasi-nil clean elements of $R$ is quasi-clean in $R$.\end{Proposition}
 \begin{proof} Let $a \in R$ be quasi-nil clean. Then $a=q+n$, where $q$ is a quasi-idempotent ($q^2=kq$  for some central unit $k$) and $n$ is nilpotent. Note that $a=(q-k)+k(1+k^{-1}n)$. Since $q^2=kq$, one checks easily that $(q-k)^2=-k(q-k)$, so $q-k$ is again a quasi-idempotent and $k(1+k^{-1}n)$ is a unit in $R$ as both $k$ and $(1+k^{-1}n)$ are units. Hence $a\in R$ is quasi-clean.\end{proof}

Let $A$ be a ring and $B$ a subring of $A$. Define

$\Re[A, B] := \{ (a_1, \ldots, a_l, b, b, \ldots) : \ a_i \in A,\, b \in B,\, l \geq 1 \},$

where addition and multiplication are defined componentwise. Then $\Re[A,B]$ forms a ring under these operations.

\begin{Proposition}
Let $R$ be a quasi-nil clean ring. Then the followings hold :

    $(i)$ $R[x]$ is not quasi-nil clean.

    $(ii)$ $R[[x]]$ is not quasi-nil clean.

    $(iii)$ Let $A$ be a ring and $B$ be a subring of $A$. Then $\Re[A,B]$ is quasi-nil clean if and only if $A$ is quasi-nil clean and for each $b \in B$, $b=q+n$, where $n \in N(B)$ and $q^2=kq$ with $k \in U(B)\cap C(B)$.
\end{Proposition}
\begin{proof} $(i)$ If $R[x]$ is quasi-nil clean then by Proposition \ref{1.4}, $R[x]$ is quasi clean which contradicts \cite[Proposition 2.8 (3)]{qc}.

$(ii)$ Suppose $R[[x]]$ is quasi-nil clean. Then by Proposition \ref{j}, we have $J(R[[x]])$ is nil. From \cite[Exercise 5.6]{lam}, it is known that $J(R[[x]])=\{ a+xf(x)~:~ a \in J(R), f(x) \in J(R[[x]]) \}$. Therefore, it is clear that $x^i\in J(R[[x]])$ for any $i \geq 2$. But $x^i$ is not nilpotent, so $J(R[[x]])$ is not nil, which is a contradiction. Therefore, $R[[x]]$ is not quasi-nil clean.

$(iii)$ Let $T = \Re[A, B]$. Assume that $T$ is quasi-nil clean. As $A$ is an image of $T$, $A$ is quasi-clean by $(i)$. Let $b \in B$ and let $\beta = (b, b, \ldots, b, b, \ldots) \in T$. Write $\beta = (e_i) + (n_i)$, where $(e_i)$ is an $(k_i)$-idempotent in $T$ and $(n_i)$ is nilpotent in $T$. We may write $(e_i) = (e_1, \ldots, e_l, e, e, \ldots)$, $(n_i) = (n_1, \ldots, n_l, n, n, \ldots)$ and $(k_i) = (k_1, \ldots, k_l, k, k, \ldots)$. Thus it follows that $u, e \in B$, $k \in U(B) \cap C(A)$, $e^2 = ke$ and $b = e + n$.

For the necessary part, let $\alpha = (a_1, \ldots, a_l, b, b, \ldots) \in T$. Then $b = e + n$, where $n \in N(B)$ and $e \in B$ with $e^2 = ke$ for some $k \in U(B) \cap C(A)$. For each $i$, write $a_i = e_i + n_i$, where $e_i$ is a $k_i$-idempotent in $A$ and $n_i \in N(A)$. Thus $(k_i) = (k_1, \ldots, k_n, k, k, \ldots)$ is a central unit in $T$ and $\alpha = (e_1, \ldots, e_n, e, e, \ldots) + (n_1, \ldots, n_l, n, n, \ldots)$ is a sum of a $(k_i)$-idempotent and a nilpotent. So $\alpha$ is quasi-nil clean.
\end{proof}

We say that an associative unital ring $R$ is unit-central (see \cite{k1}) if $U(R) \subseteq C(R)$, i.e. In a unit-central ring $R$ every invertible elements of $R$ are in the center of $R$.

\begin{Proposition} \label{p1.6}
Let $R$ be a unit central ring and $I$ be a nil ideal of $R$. Then $R$ is a quasi-nil clean ring if and only if $R/I$ is a quasi-nil clean ring. \end{Proposition}
\begin{proof}
   First, assume that $R/I$ is a quasi-nil clean ring. Let $r \in R$. Then there exists a central unit $\overline{k}$,  idempotent $\overline{e}$ and a nilpotent $\overline{n}$ of $R/I$ such that $\overline{r}=\overline{ke}+ \overline{n}$. As $I$ is nil, idempotents and units of $R/I$ can be lifted to $R$, thus we may assume $k \in U(R)$ and $e \in R$ is idempotent. Hence $r=ke+n+b$ for some $b \in I$ and then $r=ke+(n+b)$ is a quasi-nil clean decomposition in $R$. Thus $R$ is a quasi-nil clean ring.
   
   Conversely, the reverse implication is immediate since $R/I$ is a homomorphic image of $R$, and the quasi-nil clean property is preserved under homomorphic images.
\end{proof}

\begin{Proposition} \label{1.4}
    Let $I$ be a right (left) ideal of $R$. Suppose $a$ is an element of $I$ such that $a=ke+n$ for some $k\in UC(R)$, $e \in Id(R)$ and $n \in N(R)$. If either $I$ is a two-sided ideal or $en=ne$, then $e,n \in I$. \end{Proposition}
\begin{proof}
    Let $p=k^{-1}n$, a nilpotent element in $R$. Then it follows that $p$ is quasi-regular in $R$. Thus there is an element $q \in R$ such that $p+q+pq=0$. Now since $a=ke+n$, we have $ea+eaq=ke+en+keq+enq=ke+ke(k^{-1}n+q+k^{-1}nq)=ke+ke(p+q+pq)=ke$.
Thus we have $ea(1+q)=ke$. Then we find that $k^{-1}ea(1+q)=e$. If $I$ is a two-sided ideal of $R$, $e \in I$ implies that $n \in I$. On the other hand, if $en=ne$, then $ae=ea$ and $k$ is a central unit in $R$. Therefore, $e=ae(1+q)k^{-1} \in I$. This implies that $n \in I$.
\end{proof}

\begin{Corollary}\label{cor1.7}
     Let $R$ be a ring and $I$ be a right (left) ideal of $R$. If $a \in I$ is strongly quasi-nil clean in $R$, then $a$ is strongly quasi-nil clean in $I$.\end{Corollary}
\begin{proof}
    Since $a \in I$ is strongly quasi-nil clean in $R$, it can be written as $a=ke+n$, where $k$ is a central unit, $e$ is an idempotent, and $n$ is a nilpotent element of $R$ satisfying $en = ne$. By Proposition~\ref{1.4}, we have $e, n \in I$. Hence, $a$ is strongly quasi-nil clean in $I$.
\end{proof}

\begin{Corollary}
  Let $R$ be a quasi-nil clean ring and $I$ be an ideal of $R$. Then $I$ is quasi-nil clean. \end{Corollary} 

\begin{Proposition}
 Let $e \in Id(R)$. If the element $a \in eRe$ is strongly quasi-nil clean in $R$, then it is strongly quasi-nil clean in $eRe$. \end{Proposition}
\begin{proof}
  Given that $a$ is strongly quasi-nil clean in $R$. Then we can write $a=kf+n$ for some central unit $k$, idempotent $f$ and nilpotent $n$ in $R$. Note that $a \in Re$ and $Re$ is a left ideal of $R$. Thus by using Corollary \ref{cor1.7}, it follows that $a$ is strongly quasi-nil clean in $Re$. Similarly, $a \in eRe$ and $eRe$ is a right ideal of $Re$. Therefore, again by using Corollary \ref{cor1.7}, we find that $a$ is strongly quasi-nil clean in $eRe$.\end{proof}

\begin{Corollary}
    If $R$ is a strongly quasi-nil clean ring and $e \in Id(R)$, then $eRe$ is a strongly quasi-nil clean ring.\end{Corollary}  

Let $J(R)$ be the Jacobson radical of $R$. Then $R$ is called a semipotent ring if each one-sided ideal $I$ not contained in $J(R)$ contains a nonzero idempotent.

\begin{Proposition} \label{p1.11}
    Let $R$ be a quasi-nil clean ring. Then $R$ is a semipotent ring. \end{Proposition}
\begin{proof}
Let $I$ be a left ideal of $R$ such that $I \not\subset J(R)$. Then there is an element $a \in I$ such that $1-a$ is not a unit in $R$. As $R$ is quasi-nil clean, $a-1=ke+n$ for some $k \in UC(R)$, $e \in Id(R)$ and $n \in N(R)$. Note that $e \neq 1$, otherwise $1-a=-k-n=-k(1+k^{-1}n)$ is a unit. Thus $a=ke+1+n$ that implies $a=ke+u$,  where $u=1+n$ is a unit in $R$. If we take $f=u^{-1}(1-e)u$, then $f$ is a non zero idempotent and $f=u^{-1}(1-e)u=u^{-1}(1-e)(a-ke)=u^{-1}(1-e)a \in I$. Hence, $R$ is a semipotent ring.\end{proof}

 A ring $R$ is said to be Dedekind finite if for any $a,b \in R$, $ab=1$ implies that $ba=1$. Equivalently, $R$ is Dedekind finite if every left or right invertible element in $R$ is invertible. All abelian rings (rings with central idempotents) are Dedekind finite rings.

\begin{Definition} \cite{qc} A ring $R$ is called a quasi-Boolean ring if every element of $R$ is quasi-idempotent.\end{Definition}

\begin{Proposition}\cite{b} \label{p1.12}
 Let $e\in Id(R)$. Then $R$ is Dedekind finite implies that $eRe$ is also Dedekind finite.\end{Proposition}

\begin{Proposition} \label{a}
If R is a quasi-nil clean abelian ring, then every nilpotent element of $R$ is contained in $J(R)$. \end{Proposition}
\begin{proof}
    Suppose for the sake of contradiction, that there exists a nilpotent element $a \in R$ such that it is not contained in $J(R)$. Consider the principal left ideal $T=Ra$, which is not contained in  $J(R)$. By Proposition \ref{p1.11}, it follows that $R$ is semipotent. Thus there is a non zero idempotent $e=ra \in T$ for some $r \in R$. Since $R$ is abelian, so $R$ is Dedekind finite. Then by Proposition \ref{p1.12}, we find that $eRe$ is Dedekind finite. Therefore, $eae$ is a unit in $eRe$ with inverse $ere$. This is a contradiction, because a nilpotent element cannot be a unit. This completes the proof. \end{proof}

A ring $R$ is said to be reduced if it contains no nonzero nilpotent elements.

\begin{Corollary} \label{c1}
If $R$ is a quasi-nil clean abelian ring with $J(R)=\{0\}$, then $R$ is reduced and hence quasi-Boolean.
\end{Corollary}

\begin{proof}
From the above argument, all nilpotent elements of $R$ must belong to $J(R)$. This implies that $R$ is a reduced ring, having only $0$ as a nilpotent element. Hence $R$ is quasi-Boolean.
\end{proof}

\begin{Definition} (see \cite{g}) An element of $R$ is said to be $2$-good if it can be expressed as a sum of two units of $R$. The ring $R$ is called $2$-good if all of its elements are the sum of two units.\end{Definition}

\begin{Proposition}
    Let $R$ be a quasi-nil clean ring and $a \in R$ such that $aR$ contains no non zero idempotents. Then $a$ is the sum of two units in $R$, i.e. $a$ is a $2$-good element in $R$. \end{Proposition}
\begin{proof}
Let $R$ be a quasi-nil clean ring and $a \in R$. We consider the element $a-1$ in $R$. Thus we can write $a-1=q+n$, where $q$ is a quasi-idempotent (i.e. $q^2=kq$  for some central unit $k$) and $n$ is nilpotent. Then $ak^{-1}q=q+(1+n)k^{-1}q$. Consequently, $a-ak^{-1}q=(1+n)(1-k^{-1}q)$ which implies that $a(1-k^{-1}q)(1+n)^{-1}=(1+n)(1-e)(1+n)^{-1}$, which is an ideampotent in $aR$. Since $aR$ contains only trivial idempotent, so $(1+n)(1-k^{-1}q)(1+n)^{-1}=0$ that implies $k^{-1}q=1 \implies q=k$. Hence $a=k+(1+n)$, sum of two units in $R$.
\end{proof}

\begin{Proposition}
Let $R$ be a quasi-nil clean ring such that $2$ is invertible in $R$. Then $R$ is a $2$-good ring. \end{Proposition}

\begin{proof} Let $a \in R$. Since $R$ is quasi-nil clean, we can express  $\frac{a}{2} = q + n$ for some quasi-idempotent $q$ and nilpotent $n$. This implies  $q^2 = kq$ for some central unit $k$. Rewriting $a$, we obtain $a = 2q + 2n$. Since $k$ is a central unit, so we can write the expression as follows :  $a = 2q - k + k + 2n = k(2k^{-1}q - 1) + k(1 + 2k^{-1}n)$. Since both $k(2k^{-1}q - 1)$ and $k(1 + 2k^{-1}n)$ are units in $R$, it follows that $a$ is the sum of two units in $R$. Thus $R$ is a $2$-good ring.  \end{proof}

\begin{Proposition}
  Let $R$ be a ring with only trivial idempotent. Then $R$ is quasi-nil clean if and only if $R$ is local and $R/J(R) \cong \mathbb Z_{2}$.    \end{Proposition}
\begin{proof}
 Let $a \in R$. Then we can write $a=q+n$ for some quasi idempotent $q$ and nilpotent $n$. Since $R$ contains only trivial idempotent, either $q=0$ or $q$ is a central unit. Therefore, either $a$ is a unit or $a-1$ is a unit. So $R$ is a local ring. Again by Proposition \ref{j}, $J(R)$ is nil. Therefore, $R/J(R)$ is a quasi-nil clean division ring. Let $x \notin J(R)$. Then $\overline{x}=\overline{1}+\overline{n}$ for some nilpotnt $\overline{n} \in R/J(R)$ which implies that $n^k \in J(R)$. Then $n^l=0$ which implies that $n \in J(R)$. Therefore, $\overline{x}=\overline{1}$ and so $R/J(R) \cong \mathbb Z_2$.

Conversely, suppose that $R$ is a local ring with $J(R)$ is nil and $R/J(R) \cong \mathbb Z_{2}$. Then for any $a\in R$, either $a \in J(R)$ or $a-1 \in J(R)$. If $a\in J(R)$, then $a$ is nilpotent. Hence it is quasi-nil clean. If $a-1 \in J(R)$, then $a-1$ is nilpotent. Thus $a=1+(a-1)$ is a quasi-nil clean decomposition. Hence $R$ is a quasi-nil clean ring. \end{proof}

\begin{Theorem}\label{thm1.19}
Let $R$ be a unit central ring. Then $R$ is a strongly quasi-nil clean ring if and only if $N(R)$ is an ideal of $R$ and the quotient ring $R/N(R)$ is a strongly quasi-nil clean ring.
\end{Theorem}

\begin{proof}
 Since $R$ is unit central, the set of nilpotents $N(R)$ is commutative.

Let $q\in R$ be a quasi-idempotent. Then $q^2=kq$ for some central unit $k \in R$.  Let $x \in N(R)$. Consider $(xq - k^{-1}xq)^2 = 0$. This shows that $xq - k^{-1}xq^2 \in N(R)$. Since $N(R)$ is commutative, we have $x(xq - k^{-1}xq^2) = (xq - k^{-1}xq^2)x$. This implies that $x^2q - k^{-1}x^2q^2 = xqx - k^{-1}xq^2x$ $\implies qx^2q - k^{-1}qxq^2q = qxqx - qxqx = 0$. Now by using the quasi-idempotent identity $q^2=kq$, we find that $qx^2q = k^{-1}qxqxq=k^{-2}qxkqxq=k^{-2}qxq^2xq=k^{-2}(qxq)^2$. Similarly, doing in the same process, we have $qx^{2^n}q = k^{-2^n + 2}(qxq)^{2^n}$ for all $n \in \mathbb{N}$.
This implies that $qxq \in N(R)$. Since $N(R)$ is closed under addition and products with elements of $R$, it follows that $xq = (xq - k^{-1}qxq) + k^{-1}qxq \in N(R)$. Similarly, $qx \in N(R)$. Now let $a \in R$. Then $a = q' + n$, where $q'$ is a quasi-idempotent and $n \in N(R)$. Then for any $x \in N(R)$, $ax = qx + nx \in N(R)$ and $xa = xq + xn \in N(R)$. Therefore, $N(R)$ is an ideal of $R$.
Moreover, since $R$ is strongly quasi-nil clean and $N(R)$ is an ideal, it follows that $R/N(R)$ is strongly quasi-nil clean.

The reverse part follows directly from Proposition \ref{p1.6}. \end{proof}

An element $a$ in $R$ is said to be unipotent if $1-a$ is nilpotent. A ring $R$ is said to be $UU$ if all its units are unipotent, i.e. $U(R) \subseteq 1 + N(R)$ implies that $1+N(R)=U(R)$.

\begin{Theorem} \label{thm2.20}
A ring $R$ is strongly nil clean if and only if the following conditions hold :

$(i)$ $R$ is a strongly quasi-nil clean ring

$(ii)$ $R$ is a $UU$-ring

$(iii)$ $N(R)$ forms an ideal of $R$.

\end{Theorem}

\begin{proof}
First suppose that $R$ is a strongly nil clean ring. Then from \cite[Theorem 3]{h2}, conditions $(i)$, $(ii)$ and $(iii)$ follow.

Conversely, assume that $R$ is a strongly quasi-nil clean ring as well as a $UU$-ring and $N(R)$ forms an ideal of $R$. Let $a \in R$. Then  we can write $a = q + n$, where $q \in R$ is a quasi-idempotent and $n \in N(R)$. Thus $q^2 = kq$ for some central unit $k \in R$. Since $R$ is a $UU$-ring, $k = 1 + x$, where $x \in N(R)$. Then $a= kk^{-1}q + n= (1 + x)k^{-1}q + n= k^{-1}q + xk^{-1}q + n= k^{-1}q + d$, where $d = xk^{-1}q + n \in N(R)$ as $N(R)$ is an ideal. Hence
$a - a^2=k^{-1}q + d - k^{-1}q + k^{-1}qd + dk^{-1}q + d^2 = d - k^{-1}qd - k^{-1}dq - d^2 \in N(R)$. Therefore, by \cite[Lemma 2.1]{c1}, $R$ is a strongly nil clean ring.\end{proof}

From the above results, we can conclude the following result :

\begin{Corollary}
 Let $R$ be a unit central ring. Then $R$ is a strongly nil clean ring if and only if $R$ is a strongly quasi-nil clean ring as well as $R$ is a  $UU$-ring.\end{Corollary}

\section{Commutative Quasi-nil Clean Rings}
\begin{Theorem}
Every localization of a commutative quasi-nil clean ring is quasi-nil clean. \end{Theorem}
\begin{proof}
Let $R$ be a commutative quasi-nil clean ring and $S$ be a multiplicative closed subset of $R$. Let $x \in S^{-1}R$. Then $x=\frac{a}{s}$ for some $a \in R$ and $s \in S$. As $R$ is quasi-nil clean, we can write $a=q+n$, where $q$ is quasi-idempotent (i.e. $q^2 = kq$ for some central unit $k \in R$) and $n$ is nilpotent. Now consider an element $a/s \in S^{-1}R$. We express :
\[
    \frac{a}{s} = \frac{q + n}{s} = \frac{q}{s} + \frac{n}{s}.
\]
Since localization preserves nilpotency, $\frac{n}{s}$ remains nilpotent in $S^{-1}R$. It remains to show that $q/s$ is quasi-idempotent. Now $(\frac{q}{s})^2=\frac{k}{s}\frac{q}{s}$, where $\frac{k}{s}$ is a central unit in $S^{-1}R$. Hence $S^{-1}R$ is quasi-nil clean.   \end{proof}

\begin{Proposition}
In a commutative ring, every strongly $\pi$-regular ring is quasi-nil clean. \end{Proposition}
\begin{proof}
    Let $R$ be a commutative strongly $\pi$-regular ring and $a \in R $. As shown in \cite[Proposition 2.5]{d}, there exists an idempotent $e \in R$ and a unit $u \in R$ such that $a = e + u$ and $eae$ is nilpotent. Then we can write $a=(1-e)a+ae=u(1-e)+eae$, where $u(1 - e)$ is quasi-idempotent and $eae$ is nilpotent. Hence $a$ is a sum of a quasi-idempotent and a nilpotent element, in $R$. Thus $R$ is a quasi-nil clean ring. \end{proof}

\begin{Proposition} \label{q}
Let $R$ be a commutative ring and $I$ be a nil ideal of $R$. Then $R$ is a quasi-nil clean ring if and only if $R/I$ is a quasi-nil clean ring. \end{Proposition}
\begin{proof}
This result directly follows from Proposition \ref{p1.6}
\end{proof}

\begin{Theorem}
Let $R$ be a commutative ring. Then $R$ is a quasi-nil clean ring if and only if $R/J(R)$ is a quasi-Boolean ring and $J(R)$ is nil. \end{Theorem}
\begin{proof}
   Suppose that $R$ is a quasi-nil clean ring. Therefore, by Proposition \ref{j}, $J(R)$ is nil.  Again by Proposition \ref{a} and Corollary \ref{c1}, it follows that $R/J(R)$ is a quasi-Boolean ring. 
   
   On the other hand, $J(R)$ is nil and $R/J(R)$ is a quasi-Boolean ring. Hence it is quasi-nil clean. Therefore, $R$ is a quasi-nil clean ring, by Proposition \ref{q}. \end{proof}

\begin{Proposition}
    Let $R$ be a commutative $UU$-ring. Then the following conditions are equivalent : 

    $(i)$ $R$ is a quasi-nil clean ring.

    $(ii)$ $R$ is a quasi-clean ring.

    $(iii)$ $R$ is a clean ring.

    $(iv)$ $R$ is a nil clean ring. \end{Proposition}
\begin{proof}
 $(i) \implies (ii)$. Let $R$ be a quasi-nil clean ring. Then for any $a \in R$, we can write $a-1=q+n$, where $q$ is a quasi idempotent and $n$ is a nilpotent in $R$. This implies that $a=q+(1+n)$. Since $R$ is a $UU$-ring, so $1+n$ is a unit in $R$. Hence $R$ is a quasi clean ring.

 $(ii) \implies (iii)$. For any $a \in R$, we have $a=ke + u$ for some central unit $k$, unit $u$ and idempotent $e$ in $R$. Then $a=ke+u=(1+n)e+u$ for some nilpotent $n$ in $R$. Therefore, $a=e+u(1+u^{-1}ne)$. This shows that $R$ is a clean ring.

 $(iii) \implies (iv)$. For any $a \in R$, there is an idempotent $e$ and a  unit $u$ in $R$ such that $a+1=u+e=1+n+e$, where $n$ is a nilpotent in $R$. Therefore, $a=n+e$ and hence $R$ is a nil clean ring.

 $(iv) \implies (i)$ is straightforward.
\end{proof}

Let $(A,B)$ be a pair of rings, $f : A \longrightarrow B$ be a ring homomorphism and $K$ be an ideal of the ring $B$. In this context, we can consider the subring $A \bowtie^f K=\{(a,f(a)+k) : a \in A, k\in K\}$ of $A \times B$. This subring is termed the amalgamation of $A$ with $B$ along the ideal $K$ with respect to $f$, introduced and studied by D'Anna, Finocchiaro and Fontana in \cite{D}. Subsequently, we establish several connections regarding quasi-nil cleanness between $A \bowtie^f B$ and the underlying ring $A$.

\begin{Lemma}
Let $A,B$ be two rings, $f : A \longrightarrow B$ be a ring homomorphism and $K$ be an ideal of the ring $B$. Then the following statements hold : 

$(i)$ If $A \bowtie^f K$ is a quasi-nil clean ring, then so are $A$ and $f(A)+K$.

$(ii)$ If $A \bowtie^f K$ is a quasi-nil clean ring, then for each $x \in K$, $x=h+l$ for some quasi-idempotent $h$ and nilpotent $l$ in $B$.\end{Lemma}

\begin{proof}
$(i)$ According to \cite[Propoaition 5.1(3)]{D}, both the rings $A$ and $f(A)+K$ are homomorphic image of the ring $A \bowtie^f B$. As quasi-nil clean rings are closed under homomorphic images, it follows that both $A$ and $f(A)+K$ are quasi-nil clean rings.

$(ii)$ Let $x \in K$. Since $A \bowtie^f B$ is a quasi-nil clean ring, for $(0, x) \in A \bowtie^f B $  we can write
$(0,x) = (q, f(q) + h) + (n, f(n) + l)$,
where $(q, f(q) + h)$ is a quasi-idempotent and $(n, f(n) + l)$ is a nilpotent in $A \bowtie^f B$. Therefore, we see that $q$ is a quasi-idempotent and $n$ is a nilpotent in $A$. Since $f$ is a ring homomorphism, it follows that $f(q)$ is a quasi-idempotent in $B$ and $f(n)$ is nilpotent in $B$. Moreover $q + n = 0$ implies that $f(q) = -f(n)$. Since $f(n)$ is nilpotent and $f(q)$ is quasi-idempotent, it follows that $f(q) = 0 = f(n)$. Hence $x = h + l$.    \end{proof}

\begin{Lemma}\label{2.8}
 Let $A,B$ be two rings, $f : A \longrightarrow B$ be a ring homomorphism and $K$ be an ideal of the ring $B$ such that $K \subseteq N(B)$. Then $A \bowtie^f K$ is a quasi-nil clean ring if and only if $A$ is a quasi-nil clean ring.\end{Lemma}

 \begin{proof} If $A \bowtie^f K$ is a quasi-nil clean ring, then so is $A$, since $A$ is a homomorphic image of $A \bowtie^f K$.
 
Conversely, suppose that $A$ is a quasi-nil clean ring and $K \subseteq N(B)$. Let $(a,f(a)+t) \in A \bowtie^f K$. Then $a=q+n$, where $q^2=kq$ for some central unit $k$ and nilpotent $n$. This implies that $f(a)=f(q)+f(n)$, where $f(q)$ is a quasi idempotent and $f(n)$ is a nilpotent in $B$. Now $(a,f(a)+t)=(q,f(q))+(n,f(n)+t)$. Now $(q,f(q))^2=(k,f(k))(q,f(q))$. So $(q,f(q))$ is quasi-idempotent in $A \bowtie^f K$. Moreover $(n,f(n)+k)$ is a nilpotent in $A \bowtie^f K$, since $k \in K \subseteq N(B)$. Therefore, $(a,f(a)+k)$ is quasi-nil clean in $A \bowtie^f K$. Hence $A \bowtie^f K$ is a quasi-nil clean ring.
 \end{proof}

\begin{Lemma}
Let $A$ be a ring, $B$ be a regular ring, $f : A \longrightarrow B$ be a ring homomorphism and $K$ be an ideal of the ring $B$. Then $A \bowtie^f K$ is a quasi-nil clean ring if and only if $A$ is a quasi-nil clean ring.\end{Lemma}
\begin{proof}
The necessity is immediate, since $A$ is a homomorphic image of $A \bowtie^f K$.
 
Conversely, let $(a,f(a)+t) \in A \bowtie^f K$, where $t \in K$. Since $A$ is a quasi-nil clean ring, we can write $a=q+n$, where $q \in A$ is quasi idempotent and $n \in A$ is nilpotent. Then $(a,f(a)+t)=(n,f(n))+(q,f(q)+t)$, where $(n,f(n)) \in N(A \bowtie^f K)$. Since $B$ is regular, from \cite[Proposition 2.5]{qc}, it follows that $B$ is a quasi-Boolean ring. Thus $(q,f(q)+t)$ is quasi-idempotent and hence $A \bowtie^f K$ is a quasi-nil clean ring. \end{proof}

Let $R$ be a ring and $M$ be an $R$-module. Then the trivial extension of $R$ by $M$ is the ring $R \propto M=R \oplus M$ with usual addition and the multiplication is given by $(r_1,m_1)(r_2,m_2)=(r_1r_2,r_1m_2+r_2m_1)$.

\begin{Corollary} \label{cor2.10}
Let $R$ be a ring, $M$ be an $R$-module and $S=R \propto M$ be the trivial extension of $R$ by $M$. Then $S$ is a quasi-nil clean ring if and only if $R$ is a quasi-nil clean ring.
\end{Corollary}
\begin{proof}
Consider a ring monomorphism $f : R \longrightarrow S$ defined by $f(r)=(r,0)$ for every $r \in R$ and the ideal $K=0 \propto M$ of $R$. Then $f(R)+K=R \propto 0 + 0 \propto M=R \propto M = S \simeq R \bowtie^f K$, by \cite[Proposition $5.1(3)$]{D} as $f^{-1}(K)=0$. Note that $K=0 \propto M \subseteq N(S)$. Then by using Theorem \ref{2.8}, it follows that $S$ is a quasi-nil clean ring if and if $R$ is a quasi-nil clean ring.\end{proof}

\begin{Example}
    Let $R=\mathbb{Z}_{12}$. Then it is a quasi-nil clean ring with the ideal $K=<6> \subseteq N(R)$. Consider any non zero $R$-module $M$ and let $S=R \propto M$. Consider a ring epimorphism $f: S \longrightarrow R$ defined by $f(r,m)=r$ for all $r \in R$ and $m \in M$. Then

    $(i)$ $S \bowtie^f K$ is quasi-nil clean.

    $(ii)$ $S \bowtie^f K$ is not nil clean.

    Since $R$ is a quasi-nil clean ring, by Corollary \ref{cor2.10}, it follows that $S$ is a quasi-nil clean ring. Note that $K \subseteq N(R)$. Therefore, by Lemma \ref{2.8}, we have $S \bowtie^f K$ is a quasi-nil clean ring as $S$ is a quasi-nil clean ring. Note that $f(S)+K=R+K=R$ which is not nil clean because $2^2-2 \notin N(R)$ using \cite[Lemma 2.1]{c1}. Therefore, by \cite[Theorem 2.1]{b1}, it follows that $S \bowtie^f K$ is not nil clean.
\end{Example}

Now we determine some conditions under which the group ring over quasi-nil clean ring is quasi-nil clean.

If $R$ is a ring and $G$ is a group, then $RG$ denotes the group ring of $G$ over $R$. The ring homomorphism $\omega : RG \longrightarrow R$ defined as $~\Sigma r_g g \longrightarrow \Sigma r_g$ is called augmentation map. The kernel Ker$(\omega)$  is called the augmentation ideal of the group ring $RG$ and is denoted by $\Delta(RG)$. It is well known that $\Delta(RG)$ is generated by the set $\{ 1-g~: g \in G \}$.

Recall that a group $G$ is said to be a locally finite group if every finitely generated subgroup of $G$ is finite. A group $G$ is called a torsion group if every element of $G$ has finite order. If $p$ is prime, $g \in G$ is called the $p$-torsion element if the order of $g$ is $p^k$ for some $k \geq 1$. The group $G$ is called $p$-group if every element of $G$ is $p$-torsion. Also recall that for a given prime $p$, we say that a finite group $G$ is an elementary $p$-group if $G$ is a direct product of finitely many copies of $\mathbb{Z}_p$, i.e. $G \cong \mathbb{Z}^m_p$ for some integer $m \geq 1$.

\begin{Proposition}
If $R$ is a quasi-nil clean ring with $2 \in U(R)$ and $G$ is a finite abelian elementary $2$-group, then the group ring $RG$ is a quasi-nil clean ring. \end{Proposition}
\begin{proof}
Since $G$ is a finite abelian elementary $2$-group, it can be written as $~G \cong \mathbb{Z}_2\times \mathbb{Z}_2 \ldots \times \mathbb{Z}_2$. Given that $2 \in U(R)$, then $R\mathbb{Z}_2 \cong R \oplus R$. Again since $2$ is invertible in $R\mathbb{Z}_2$, $R(\mathbb{Z}_2 \times \mathbb{Z}_2) \cong R\mathbb{Z}_2 \oplus R\mathbb{Z}_2 \cong R \oplus R \oplus R \oplus R$. Inductively, we get $RG \cong R(\mathbb{Z}_2\times \mathbb{Z}_2 \ldots \times \mathbb{Z}_2) \cong R \oplus R\oplus R \ldots \oplus R $. Since $R$ is a quasi-nil clean ring, so the group ring $RG$ is a quasi-nil clean ring.
\end{proof}

\begin{Proposition}
Let $R$ be a commutative ring of characteristic $2^k$ ($k\in \mathbb{N}$) and $G$ be an abelian group. If $R$ is a quasi-nil clean ring and $G$ is a $2$-torsion group, then the group ring $RG$ is a quasi-nil clean ring. \end{Proposition}
\begin{proof}
 Suppose that $R$ is a quasi-nil clean ring and $G$ is a $2$-torsion group. Since the characteristic of $R$ is $2^k$, we have $2^{k} = 0$ in $R$, which implies that $2$ is nilpotent in $R$. By \cite[Lemma~2.5]{m1}, it follows that $1 - g$ is nilpotent in $RG$ for every $g \in G$. The augmentation ideal $\Delta (RG)$ is generated by the set $\{ 1-g~: g \in G \}$ and so it is nil. Then the result follows from Proposition \ref{q}.
\end{proof}

\begin{Theorem}
Let $R$ be a commutative ring, $p$ be a prime such that $p \in J(R)$ and $G$ be an abelian group. If $G$ is a locally finite $p$-group, then the group ring $RG$ is a quasi-nil clean ring if and only if $R$ is a quasi-nil clean ring. \end{Theorem}
\begin{proof}
    If the group ring $RG$ is a quasi-nil clean ring, then so is the ring $R$, as $R$ is a homomorphic image of $RG$.

    Conversely, suppose that $R$ is a quasi-nil clean ring with $p \in J(R)$. Let $\omega : RG \longrightarrow R$ be a ring homomorphism with an augmentation ideal $\Delta(RG)$. Since $G$ is a abelian $p$-group and $p \in J(R) $, so by \cite[Theorem 2.1]{gr1}, it follows that $\Delta(RG) \subseteq J(RG)$. Since $R$ is a quasi-nil clean ring, by Proposition \ref{j} $J(R)$ is nil. Then $p$ is a nilpotent in $R$. Since $G$ is locally finite $p$-group and $p$ is nilpotent in $R$, by \cite[Proposition $16(ii)$]{c2}, it follows that $\Delta(RG)$ is nil. Now $RG/\Delta(RG) \cong R$. Thus it follows that $RG/\Delta(RG)$ is a quasi-nil clean ring. Hence, $RG$ is a quasi-nil clean ring.
\end{proof}

\noindent
\textbf{\large Funding}

\noindent
The research of the first author is funded by  University Grants Commission (UGC) (Award Letter No.: 211610081351), Govt. of India.
\vspace{0.5cm}

\noindent
\textbf{\large Declarations} 

\noindent
\textbf{Conflict of interest:} All authors declare that they have no conflicts of interest.

\end{document}